\documentclass{article}
\usepackage{arxiv}
\usepackage[utf8]{inputenc} 
\usepackage[T1]{fontenc}    
\usepackage{hyperref}       
\usepackage{url}            
\usepackage{booktabs}       
\usepackage{amsfonts}       
\usepackage{nicefrac}       
\usepackage{microtype}      
\usepackage{lipsum}		
\usepackage{graphicx}
\usepackage{natbib}
\usepackage{doi}
\usepackage{amsmath, amsthm, amssymb}
\newcommand{\seq}[1]{\left<#1\right>}
\newcommand{\norm}[1]{\left\Vert#1\right\Vert}
\newcommand{\abs}[1]{\left\vert#1\right\vert}

\newcommand{\bra}[1]{\left(#1\right)}
\newcommand{\R}{\mathbb R}
\newcommand{\N}{\mathbb N}
\renewcommand{\to}{\longrightarrow}
\newcommand{\h}{\mathfrak{H}}
\newcommand{\bh}{\mathfrak{B}(\mathfrak{H})}

  \renewcommand{\S}{\mathfrak{S}}
  \newcommand{\T}{\mathfrak{T}}

 \newcommand{\A}{\mathfrak{A}}
 \newcommand{\B}{\mathfrak{B}}
 \newcommand{\C}{\mathfrak{C}}
 \newcommand{\D}{\mathfrak{D}}
 \newcommand{\X}{\mathfrak{X}}
\newtheorem{theorem}{Theorem}[section]
\newtheorem{lemma}[theorem]{Lemma}

\newtheorem{corollary}[theorem]{Corollary}

\newtheorem{example}[theorem]{Example}

\newtheorem{remark}[theorem]{Remark}

\newcommand\mystyle{\everymath{\displaystyle}}
\mystyle
\title{An estimate for the numerical radius
of the Hilbert space operators and a numerical radius inequality}

\author{ \href{https://orcid.org/0000-0000-0000-0000}{\includegraphics[scale=0.06]{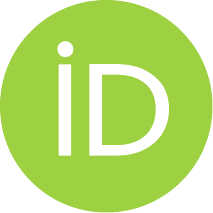}\hspace{1mm}M.H.M. Rashid}\thanks{Corresponding Author} \\
	Dept. of Maths \& Stat.\\
	Faculty of Science P.O.Box(7)\\
Mu'tah University-Al-karak-Jordan \\
	\texttt{mrash@mutah.edu.jo} \\
	\And
	\href{https://orcid.org/0000-0000-0000-0000}{\includegraphics[scale=0.06]{orcid.pdf}\hspace{1mm}Feras Bani-Ahmad} \\
	Dept. of Maths \\
 Faculty of Science, The Hashemite University P.O.Box 330127\\
Zarqa 13133, Jordan\\
	\texttt{fbaniahmad@hu.ed.jo} \\
}

\renewcommand{\shorttitle}{\textit{arXiv} Template}

\hypersetup{
pdftitle={An estimate for the numerical radius},
pdfsubject={q-bio.NC, q-bio.QM},
pdfauthor={M.H.M.Rashid},
pdfkeywords={numerical radius, inner product space, $C^*$-algebra, $A$-module},
}

\begin{document}
\maketitle

\begin{abstract}
	We provide a number of sharp inequalities involving the usual operator norms of Hilbert
 space operators and powers of the numerical radii. Based on the traditional convexity
 inequalities for nonnegative real numbers and some
generalize earlier numerical radius inequalities, operator.
Precisely, we prove that if  $\A_i,\B_i,\X_i\in\bh$ ($i=1,2,\cdots,n$), $m\in\N$, $p,q>1$ with $\frac{1}{p}+\frac{1}{q}=1$ and $\phi$ and $\psi$
 are non-negative functions on $[0,\infty)$ which are continuous such that $\phi(t)\psi(t)=t$ for all $t \in [0,\infty)$, then
\begin{equation*}
    w^{2r}\bra{\sum_{i=1}^{n}\X_i\A_i^m\B_i}\leq \frac{n^{2r-1}}{m}\sum_{j=1}^{m}\norm{\sum_{i=1}^{n}\frac{1}{p}S_{i,j}^{pr}+\frac{1}{q}T_{i,j}^{qr}}-r_0\inf_{\norm{x}=1}\rho(\xi),
  \end{equation*}
  where $r_0=\min\{\frac{1}{p},\frac{1}{q}\}$, $S_{i,j}=\X_i\phi^2\bra{\abs{\A_i^{j*}}}\X_i^*$, $T_{i,j}=\bra{\A_i^{m-j}\B_i}^*\psi^2\bra{\abs{\A_i^j}}\A_i^{m-j}\B_i$ and
  $$\rho(x)=\frac{n^{2r-1}}{m}\sum_{j=1}^{m}\sum_{i=1}^{n}\bra{\seq{S_{i,j}^r\xi,\xi}^{\frac{p}{2}}-\seq{T_{i,j}^r\xi,\xi}^{\frac{q}{2}}}^2.$$
\end{abstract}


\section{Introduction}
Let $\h$ be complex Hilbert space and $\bh$ be the $C^*$-algebra   of all bounded linear operator on $\h$.
An operator $\T\in\bh$ is said to be {\it positive} if $\seq{\T\xi,\xi}\geq 0$ holds for all $\xi\in\h$. We write $\T\geq 0$ if $\T$ is positive. \\
\indent The numerical radius of $\T\in\bh$ is defined by
$$w(\T)=\sup\{|\lambda|:\lambda\in W(\T)\}=\sup\{|\seq{\T\xi,\xi}|:\xi\in\h,\norm{x}=1\}.$$
It is well known that $w(\cdot)$ defines a norm on $\bh$, which is equivalent to the usual operator norm $\norm{\cdot}$.
In fact, for any $\T\in \bh$,
\begin{equation}\label{N1}
  \frac{1}{2}\norm{\T}\leq w(\T)\leq \norm{\T}.
\end{equation}
 Also, if $\T\in \bh$ is normal, then $w(\T)=\norm{\T}$.\\
 An important inequality for $w(\T)$ is the power
inequality stating that $w(\T^n)\leq (w(\T))^n$ for every natural numbers $n$.\\
Several numerical radius inequalities improving the inequalities in (\ref{N1}) have been
recently given in \cite{Drag1, Drag2, kit1, Rashid, Rash1, Rash2, Rash3}. For instance, \cite{kit2, kit3} proved that for any $\A\in\bh$,
\begin{equation}\label{Muna6}
  w(\A)\leq \frac{1}{2}\norm{|\A|^2+|\A^*|}\leq \frac{1}{2}\left(\norm{\A}+\norm{\A^2}^{1/2}\right).
\end{equation}
where $|\A|=\sqrt{\A^*\A}$ is the absolute value of $\A$, and
\begin{equation}\label{Muna7}
 \frac{1}{4}\norm{\A^*\A+\A\A^*}\leq w^2(\A)\leq \frac{1}{2}\norm{\A^*\A+\A\A^*}.
\end{equation}
Also, in the same paper, it was shown that
\begin{equation}\label{Muna7.5}
  \norm{\A+\B}^2\leq \norm{|\A|^2+|\B|^2}+\norm{|\A^*|^2+|\B^*|^2}.
\end{equation}
\cite{KH} established the generalizations of inequality (\ref{Muna6}) and the second inequality (\ref{Muna7})
as follows:
\begin{equation}\label{Muna8}
  w^r(\A)\leq \frac{1}{2}\norm{|\A|^{2r\lambda}+|\A^*|^{2r(1-\lambda)}}
\end{equation}
and
\begin{equation}\label{Muna8}
  w^{2r}(\A)\leq \norm{\lambda|\A|^{2r}+(1-\lambda)|\A^*|^{2r}},
\end{equation}
where $0<\lambda<1$ and $r\geq 1$.\\
A general numerical radius inequality has been established by  \cite{kit3}, it has been proved that
if $\A,\B,\C,\D,\T,\S\in\bh$, then
\begin{equation}\label{Muna9}
  w(\A\T\B+\C\S\D)\leq \frac{1}{2}\norm{\A|\T^*|^{2(1-\alpha)}\A^*+\B^*|\T|^{2\alpha}\B+\C|\S^*|^{2(1-\alpha)}\C^*+\D^*|\S|^{2\alpha}\D}
\end{equation}
for all $\alpha\in(0,1)$.\\
\indent Although several open problems relating to numerical radius inequalities for bounded linear operators remain unsolved, work on establishing numerical radius inequalities for a number of bounded linear operators has begun (see, for example, \cite{GR} and \cite{Rashid, Rash1, Rash2, Rash3}).
If $\A,\B\in \bh$, then
$$
  w(\A\B)\leq 4w(\A)w(\B).
$$
In the case that $NM=MN$, we have
$$
  w(\A\B)\leq 2w(\A)w(\B).
$$
  Moreover, if $N$ and $M$ are normal, then
  $$w(\A\B)\leq w(\A)w(\B).$$
Recently,  ~\cite{Drag3} proved that if $N,M\in\bh$ and $r\geq 1$, then
\begin{equation}\label{watfa1}
  w^r(\B^*\A)\leq \frac{1}{2}\norm{|\A|^{2r}+|\B^*|^{2r}}.
\end{equation}
\cite{AA} discovered a fascinating numerical radius inequality, it has
been shown that if $\A,X,\B \in \bh$, then
\begin{equation}\label{shab1}
  w^{r}(\A^*X\B)\leq \frac{1}{2}\norm{(\A^*|X^*|^{2\nu}\A)^r+(\B^*|X|^{2(1-\nu)}\B)^r},\,\,r\geq 1, 0<\nu<1.
\end{equation}
Very recently,  \cite{Dolat}, showed that if
$\A_i,\B_i,\X_i\in\bh$ ($i=1,2,\cdots,n$), $m\in\N$ and $\phi$ and $\psi$
 are non-negative functions
on $[0,\infty)$ which are continuous such that $\phi(t)\psi(t)=t$ for all $t \in [0,\infty)$, then
\begin{equation}\label{watfa2}
  w^{r}\bra{\sum_{i=1}^{n}\X_i\A_i^m\B_i}\leq \frac{n^{2r-1}}{2m}\sum_{j=1}^{m}\norm{\sum_{i=1}^{n}(E_{i,j})^{r}+(W_{i,j})^{r}}
\end{equation}
for $r\geq 1$, where $E_{i,j}=\X_i \phi^2\bra{\abs{\A_i^{j*}}}\X_i^*$ and $W_{i,j}=(\A_i^{m-j}\B_i)^*\psi^2\bra{\abs{\A_i^j}}\A_i^{m-j}\B_i$.\\
\indent The goal of this study is to develop significant extensions of these inequalities based on the classic convexity inequalities for nonnegative real numbers and some operator inequalities. For the sum of two operators, usual operator norm inequalities and a related numerical radius inequality are also provided. In specifically, if $i = 1, 2, \cdots n \in \N$, $\A_i, \B_i$, and $\X_i$ are bounded linear operators, we estimate the numerical radius to  $\sum_{j=1}^{m} \X_i\A_i^m\B_i$ for some $m \in \N$.

\section{Inequalities for sums and products of operators}
In this part, we built a generic numerical radius inequality for Hilbert space operators that results in well-known new numerical radius inequalities as an example. A norm inequality operator sets off this section. In fact, we give $\norm{\B^*\A+\D\C^*}$ an additional upper bound.
However, the following lemma is crucial to the theorem's proof.
\begin{lemma}[\cite{Drag0}]
  Let $\xi,\zeta,\eta\in \h$. Then we have
  \begin{equation}\label{in-prod}
    \seq{\eta,\xi}^2+\seq{\eta,\zeta}^2 \leq \norm{\eta}^2\max\{\norm{\xi}^2,\norm{\zeta}^2\}+|\seq{\xi,\zeta}|.
  \end{equation}
\end{lemma}
\begin{theorem}\label{thmA1}
  Let $\A,\B,\C,\D\in\bh$. Then
  \begin{eqnarray*}
    \norm{\B^*\A+\D\C^*}^2 &\leq &\frac{1}{2}\left[\norm{|\A^*\B|^2+|\C\D^*|^2}+\norm{|\A^*\B|^2-|\C\D^*|^2}\right] \\
     &+&w(\C\D^*\A^*\B)+2\norm{\B^*\A}\norm{\D\C^*}.
  \end{eqnarray*}
\end{theorem}
\begin{proof} For $\xi,\zeta\in\h$, we have by triangle inequality, we have
\begin{eqnarray}\label{M1}
   |\seq{(\B^*\A+\D\C^*)\xi,\zeta}^2 &\leq& |\seq{\B^*\A\xi,\zeta}|^2+|\seq{\D\C^*\xi,\zeta}|^2 \\
  &+&  +2|\seq{\B^*\A\xi,\zeta}\seq{\D\C^*\xi,\zeta}|.\nonumber
\end{eqnarray}
Now  in inequality (\ref{in-prod}), for all $\xi,\zeta\in \h$, letting $\xi=\A^*\B\zeta$, $\zeta=\C\D^*\zeta$ and
  $\eta=\xi$ with $\norm{\xi}=\norm{\zeta}=1$, we get
  $$|\seq{\xi,\A^*\B\zeta}|^2+|\seq{\xi,\C\D^*\zeta}|^2\leq \max\{\norm{\A^*\B\zeta}^2,\norm{\C\D^*\zeta}^2\}+|\seq{\A^*\B\zeta,\C\D^*\zeta}|.$$
  Now use the fact that
  $$\max\{\sigma,\tau\}=\frac{1}{2}\left[\sigma+\tau+|\sigma-\tau|\right]\quad \mbox{for any}\quad \sigma,\tau\in\R, $$
   we have
   \begin{eqnarray}\label{M2}
     |\seq{\B^*\A\xi,\zeta}|^2+|\seq{\D\C^*\xi,\zeta}|^2 &\leq&\frac{1}{2}\left[\seq{\left |\A^*\B|^2+|\C\D^*|^2\right)\zeta,\zeta}\right.\\
      &+& \left.\left|\seq{\left(|\A^*\B|^2-|\C\D^*|^2\right)\zeta,\zeta}\right|+\left|\seq{\A^*\B\zeta,\C\D^*\zeta}\right|\right].\nonumber
   \end{eqnarray}
   Combining the inequalities (\ref{M1}) and (\ref{M2}), we have
   \begin{eqnarray*}
     |\seq{(\B^*\A+\D\C^*)\xi,\zeta}^2 &\leq &\frac{1}{2}\left[\seq{\left |\A^*\B|^2+|\C\D^*|^2\right)\zeta,\zeta}\right.\\
      &+& \left.\left|\seq{\left(|\A^*\B|^2-|\C\D^*|^2\right)\zeta,\zeta}\right|+\left|\seq{\A^*\B\zeta,\C\D^*\zeta}\right|\right].\\
     &+&2\left|\seq{\B^*\A\zeta,\zeta}\seq{\C\D^*\zeta,\zeta}\right|.
   \end{eqnarray*}
   Taking the supremum over all unit vectors $\xi,\zeta$, we obtain the desired inequality.
\end{proof}
In  Theorem \ref{thmA1}, if we let $\A=\B=\C=\D=\S$, we have
\begin{corollary}
  Let $\S\in\bh$. Then
  \begin{eqnarray*}
    \norm{\S^*\S+\S\S^*}^2 &\leq&\frac{1}{2}\left[\norm{|\S|^4+|\S^*|^4}+\norm{|\S|^4-|\S^*|^4}\right] \\
     &+&w(|\S^*|^4)+2\norm{|\S|^2}\norm{|\S^*|^2}.
  \end{eqnarray*}
\end{corollary}
In the proof of Theorem \ref{thmA1}, if we let $\xi=\zeta$, we have
\begin{corollary}
  Let $\A,\B,\C,\D\in\bh$. Then
  \begin{eqnarray*}
    w^2(\B^*\A+\D\C^*) &\leq&\frac{1}{2}\left[\norm{|\A^*\B|^2+|\C\D^*|^2}+\norm{|\A^*\B|^2-|\C\D^*|^2}\right] \\
     &+&w(\D\C^*\A^*\B)+2w(\B^*\A)w(\D\C^*).
  \end{eqnarray*}
\end{corollary}
The following lemma gives a basic but useful extension for four operators of the Schwarz
inequality due to Dragomir \cite{Drag4}.
\begin{lemma}\label{D1}
  Let $\A,\B,\C,\D\in\bh$. Then for $\xi,\zeta\in\h$ we have the inequality
  \begin{equation*}
    |\seq{\D\C\B\A\xi,\zeta}|^2\leq \seq{|\B\A|^2\xi,\xi}\seq{|(\D\C)^*|^2\zeta,\zeta}
  \end{equation*}
  The equality case holds if and only if the vectors $\B\A\xi$ and $C^*D^*\zeta$ are linearly dependent in $\h$.
\end{lemma}
The following lemma, known as the H\"older-McCarthy inequality, is a well-known
 conclusion derived from Jensen's inequality and the spectral theorem for positive operators (see \cite{kit1}).
\begin{lemma}\label{Holder}
  Let $\T\in\bh$, $\T\geq 0$  and let $\xi\in\h$ be any unit vector. Then  we have
  \begin{enumerate}
    \item [(i)] $\seq{\T\xi,\xi}^r\leq \seq{\T^r\xi,\xi}$ for $r\geq 1$.
    \item [(ii)] $\seq{\T^r\xi,\xi}\leq \seq{\T\xi,\xi}^r$ for $0<r\leq 1$.
    \item [(iii)] If $\T$ is invertible, then $\seq{\T\xi,\xi}^r\leq \seq{\T^r\xi,\xi}$  for all $r<0$.
  \end{enumerate}
\end{lemma}
The next result is well known in the literature as the Mond-Pe\v cari\'c inequality \cite{MP}.
\begin{lemma}\label{yahya1}
  If $\psi$ is a convex function on a real interval $J$ containing the spectrum of the
self-adjoint operator $\T$, then for any unit vector $\xi\in\h$,
\begin{equation}\label{MP}
  \psi(\seq{\T\xi,\xi})\leq \seq{\psi(\T)\xi,\xi}
\end{equation}
and the reverse inequality holds if $\psi$ is concave.
\end{lemma}
The forth  lemma is a direct consequence of \cite[Theorem 2.3]{AS}.
\begin{lemma}\label{yahya2}
  Let $\psi$ be a non-negative non-decreasing convex function on $[0,\infty)$ and let $\T,\S\in \h$
   be positive operators. Then for any $0<\mu<1$,
   \begin{equation}\label{AS}
     \norm{\psi(\mu \T+(1-\mu)\S)}\leq \norm{\mu \psi(\T)+(1-\mu)\psi(\S)}.
   \end{equation}
\end{lemma}
The above four lemmas admit the following more general result.
\begin{theorem}\label{Rahma1}
  Let $\A,\B,\C,\D\in\bh$. If $\psi$ is a non-negative increasing convex function on $[0,\infty)$, then
  for any $0<\mu<1$,
  \begin{equation}\label{T1}
    \psi(w^2(\D\C\B\A))\leq \norm{\mu \psi\left(|\B\A|^{\frac{2}{\mu}}\right)+(1-\mu)  \psi\left(|(\D\C)^*|^{\frac{2}{1-\mu}}\right)}.
  \end{equation}
  In particular,
  \begin{equation}\label{T2}
    w^{2r}(\D\C\B\A)\leq \norm{\mu |\B\A|^{\frac{2r}{\mu}}+(1-\mu)  |(\D\C)^*|^{\frac{2r}{1-\mu}}}
  \end{equation}
  for all $r\geq 1$.
\end{theorem}
\begin{proof}
  For any unit vector $\xi\in\h$, we have
  \begin{eqnarray*}
    |\seq{\D\C\B\A\xi,\xi}|^2 &\leq& \seq{|\B\A|^2\xi,\xi}\seq{|(\D\C)^*|^2\xi,\xi}\quad(\mbox{by (\ref{D1})}) \\
     &\leq&  \seq{|\B\A|^{\frac{2}{\mu}}\xi,\xi}^{\mu}\seq{|(\D\C)^*|^{\frac{2}{1-\mu}}\xi,\xi}^{1-\mu}\\
     &&\quad(\mbox{by Lemma \ref{MP} for concavity of}\,\,\psi(t)=t^{\mu}\,\,\mbox{for}\,\, 0<\mu<1)\\
     &\leq&\mu \seq{|\B\A|^{\frac{2}{\mu}}\xi,\xi}+(1-\mu)\seq{|(\D\C)^*|^{\frac{2}{1-\mu}}\xi,\xi}\\
     &&\quad(\mbox{by weighted arithmetic-geometric mean inequality }).
  \end{eqnarray*}
  Taking the supremum over $\xi\in\h$ with $\norm{\xi}=1$, we infer that
  \begin{equation}\label{T3}
    w^2(\D\C\B\A)\leq \norm{\mu |\B\A|^{\frac{2}{\mu}}+(1-\mu)|(\D\C)^*|^{\frac{2}{1-\mu}}}.
  \end{equation}
  On account of assumptions on $\psi$, we can write
  \begin{eqnarray*}
    \psi(w^2(\D\C\B\A)) &\leq &\psi\left(\norm{\mu |\B\A|^{\frac{2}{\mu}}+(1-\mu)|(\D\C)^*|^{\frac{2}{1-\mu}}}\right) \\
     &\leq& \norm{\mu \psi\left(|\B\A|^{\frac{2}{\mu}}\right)+(1-\mu)\psi\left(|(\D\C)^*|^{\frac{2}{1-\mu}}\right)}\quad(\mbox{by Lemma \ref{AS}}).
  \end{eqnarray*}
 The inequality (\ref{T2}) follows directly from (\ref{T1}) by taking $\psi(t)=t^r$ ($r\geq 1$).
\end{proof}
In the following result, we want to improve (\ref{shab1}) under certain mild situations. We'll need the arithmetic-geometric mean inequality refinement \cite{TMFH} to do this.
\begin{lemma}\label{Mohd1}
  Suppose that $\mu,\nu>0$ and positive real numbers $\delta,\Delta$ satisfy
  $$\min\{\mu,\nu\}\leq \delta<\Delta\leq \max\{\mu,\nu\}.$$
  Then
  $$\frac{\Delta+\delta}{2\sqrt{\delta\Delta}}\sqrt{\mu\nu}\leq \frac{\mu+\nu}{2}.$$
\end{lemma}
The following lemma is very useful in the proof of the next result.
\begin{lemma}\label{Mohd2}
  Let $\psi$ be a non-negative increasing convex function on $[0,\infty)$, $\psi(0)=0$ and $\alpha\in [0,1]$.
  Then $\psi(\alpha t)\leq \alpha \psi(t)$.
\end{lemma}
\begin{theorem}\label{Prof1}
   Let $\A,\B,\C,\D\in\bh$ and let $\psi$ be a non-negative increasing convex function on $[0,\infty)$.
   If
   $$0<|\B\A|^2\leq \delta<\Delta\leq |(\D\C)^*|^2$$
   or
  $$0<|(\D\C)^*|^2\leq \delta<\Delta\leq |\B\A|^2,$$
  then
  \begin{equation}\label{S1}
    \psi(w(\D\C\B\A))\leq \frac{\sqrt{\delta\Delta}}{\delta+\Delta}\norm{\psi\left(|\B\A|^2\right)+\psi\left(|(\D\C)^*|^2\right)}.
  \end{equation}
\end{theorem}
\begin{proof}
  It follows from Lemma \ref{D1} that
  \begin{equation}\label{S2}
    |\seq{\D\C\B\A\xi,\xi}| \leq \sqrt{\seq{|\B\A|^2\xi,\xi}\seq{|(\D\C)^*|^2\xi,\xi}}.
  \end{equation}
  \begin{eqnarray}\label{S3}
     &\leq&\frac{\sqrt{\Delta\delta}}{\Delta+\delta}\left[\seq{|\B\A|^2\xi,\xi}+\seq{|(\D\C)^*|^2\xi,\xi}\right] \\
     &=& \frac{\sqrt{\Delta\delta}}{\Delta+\delta}\seq{\left(|\B\A|^2+|(\D\C)^*|^2\right)\xi,\xi}.\nonumber
  \end{eqnarray}
  Combining (\ref{S2}) and (\ref{S3}), we obtain
  \begin{equation}\label{S4}
     |\seq{\D\C\B\A\xi,\xi}|\leq \frac{\sqrt{\Delta\delta}}{\Delta+\delta}\seq{\left(|\B\A|^2+|(\D\C)^*|^2\right)\xi,\xi}
  \end{equation}
  Taking the supremum over $\xi\in\h$ with $\norm{x}=1$, we infer that
  $$w(\D\C\B\A)\leq \frac{\sqrt{\Delta\delta}}{\Delta+\delta}\norm{|\B\A|^2+|(\D\C)^*|^2}.$$
  Now, since $f$ is a non-negative increasing convex function, we have
  \begin{eqnarray*}
    \psi(w(\D\C\B\A)) &\leq &\psi\left(\frac{2\sqrt{\Delta\delta}}{\Delta+\delta}\norm{\frac{|\B\A|^2+|(\D\C)^*|^2}{2}}\right)\\
     &\leq&\frac{2\sqrt{\Delta\delta}}{\Delta+\delta}\psi\left(\norm{\frac{|\B\A|^2+|(\D\C)^*|^2}{2}}\right)\\
     &&\,\,\,(\mbox{by Lemma \ref{Mohd2} because}\,\, \alpha=\frac{2\sqrt{\Delta\delta}}{\Delta+\delta}\leq 1)\\
     &\leq& \frac{2\sqrt{\Delta\delta}}{\Delta+\delta}\norm{\psi\left(\frac{|\B\A|^2+|(\D\C)^*|^2}{2}\right)} \\
    &\leq& \frac{\sqrt{\Delta\delta}}{\Delta+\delta}\norm{\psi\left(|\B\A|^2\right)+\psi\left(|(\D\C)^*|^2\right)}\,\,\,(\mbox{by Lemma \ref{yahya2}}).
  \end{eqnarray*}
\end{proof}
As an applications of Theorem \ref{Prof1}, we have
\begin{corollary}
   Let $\T\in\bh$, $\alpha,\beta\geq 0$ with $\alpha+\beta\geq 1$ and let $\psi$ be a non-negative increasing convex function on $[0,\infty)$.
   If
   $$0<|\T|^{2\beta}\leq \delta<\Delta\leq |\T^*|^{2\alpha}$$
   or
  $$0<|\T^*|^{2\alpha}\leq \delta<\Delta\leq |\T|^{2\beta},$$
  then
  \begin{equation}\label{S1}
    \psi(w(\T|\T|^{\beta-1}\T|\T|^{\alpha-1}))\leq \frac{\sqrt{\Delta\delta}}{\Delta+\delta}\norm{\psi\left(|\T|^{2\beta}\right)+\psi\left(|\T^*|^{2\alpha}\right)}.
  \end{equation}
\end{corollary}
\begin{remark}\label{Rahma23} Following (\ref{S1}) we list here some particular inequalities of interest.
\begin{enumerate}
  \item [(i)] If we let $\psi(t)=t^r$ ($r\geq 1$), we have
  $$w^r(\D\C\B\A)\leq \frac{\sqrt{\Delta\delta}}{\Delta+\delta}\norm{|\B\A|^{2r}+|(\D\C)^*|^{2r}},$$
  whenever
  $$0<|\B\A|^2\leq \delta<\Delta\leq |(\D\C)^*|^2\quad\mbox{or}\quad 0<|(\D\C)^*|^2\leq \delta<\Delta\leq |\B\A|^2.$$
  \item [(ii)] Letting $\D=\S^*,\A=\T$ and let $\psi(t)=t^r$ ($r\geq 1$), we have
     $$w^r(\T^*\S)\leq \frac{\sqrt{\Delta\delta}}{\Delta+\delta}\norm{|\T|^{2r}+|\S|^{2r}},$$
     whenever
  $$0<|\T|^2\leq \delta<\Delta\leq |\S|^2\quad\mbox{or}\quad 0<|\S|^2\leq \delta<\Delta\leq |\T|^2.$$
  \item [(iii)] Letting $\C=\D=\B=I$ and $\A=\T$ and let $\psi(t)=t^r$ ($r\geq 1$), we have
  $$w^r(\T)\leq \frac{\sqrt{\Delta\delta}}{\Delta+\delta}\norm{|\T|^{2r}+I}, $$
  whenever
  $$0<|\T|^2\leq \delta<\Delta\leq MI\quad\mbox{or}\quad 0<I\leq \delta<\Delta\leq |\T|^2.$$
\end{enumerate}
\end{remark}
We give an example to clarify part (ii) in Remark \ref{Rahma23}
\begin{example} Let $\S=\begin{bmatrix}3/2 & 1/2 \\1/2& 3/2 \\ \end{bmatrix}$ and $\T=\begin{bmatrix}1/2 & 0 \\0 & 1/2 \\ \end{bmatrix}$ and $r=2$. A simple
calculations show that  $\S^*\T=\begin{pmatrix}\frac{3}{4}&\frac{1}{4}\\ \frac{1}{4}&\frac{3}{4}\end{pmatrix}$ and so $w^2(\S^*\T)=1$,
$\norm{|\S|^4+|\T|^4}=\frac{257}{16}=16.0625$. If we take $\delta=0.3$ and $\Delta=.4$, then
$$w^2(\S^*\T)=1\leq \frac{\sqrt{\Delta\delta}}{\Delta+\delta}\norm{|\S|^4+|\T|^4}=7.94$$
\end{example}
Recall that the weighted operator arithmetic mean $\nabla_{\nu}$ and geometric mean $\sharp_{\nu}$, for $0 <\nu < 1$,
positive invertible operator $\A$, and positive operator $\B$, are defined as follows:
$$\A\nabla_{\nu} \B=\nu \A+(1-\nu) \B\quad\mbox{and}\quad \A\sharp_{\nu}\B=\A^{\frac{1}{2}}(\A^{-\frac{1}{2}}\B\A^{-\frac{1}{2}})^{\nu}\A^{\frac{1}{2}}.$$
If $\nu=\frac{1}{2}$ we denote the arithmetic and geometric means, respectively, by $\nabla$ and $\sharp$.
\begin{theorem}\label{Boshra22}
  Let $\A,\B,\C,\D\in\bh$,  and let $f$ be a non-negative
increasing convex function on $[0,\infty)$. If for given $m',M' > 0$,
$$0<m'\leq |\B\A|^2\leq |(\D\C)^*|^2\leq M'\quad\mbox{or}\,\,0<m'\leq |(\D\C)^*|^2\leq |\B\A|^2\leq M',$$
then
$$h(w(\D\C\B\A))\leq \frac{1}{2\gamma}\norm{\psi\left(|\B\A|^2\right)+\psi\left(|(\D\C)^*|^2\right)},$$
where
$$\gamma:=\left(1-\frac{1}{8}\left(1-\frac{1}{h'}\right)^2\right)^{-1}\geq 1\,\,\,\mbox{with}\,\, h'=\frac{M'}{m'}.$$
\end{theorem}
To prove Theorem \ref{Boshra22}, we need the following result that established by  \cite{Furu}.
\begin{corollary}\label{Boshra33}
  Let $\nu\in [0,1]$, $r_1\in [-1,0)$, $r_2\in (0,1]$ and let $\T$ and $\S$ be strictly positive
operators satisfying (i) $0 < m \leq \T \leq m' < M' \leq \S \leq M$ or (ii) $0 < m \leq \S \leq m' < M' \leq \T \leq M$
 with $h=\frac{M}{m}$ and $h'=\frac{M'}{m'}$. Then
 $$exp_{r_1}\left(\frac{\nu(1-\nu)}{2}\left(\frac{h-1}{h}\right)^2\right)\T\sharp_{\nu}\S
 \leq \T\nabla_{\nu}\S\leq exp_{r_2}\left(\frac{\nu(1-\nu)}{2}\left(h'-1\right)^2\right)\T\sharp_{\nu}\S. $$
\end{corollary}
\begin{proof}[Proof of Theorem \ref{Boshra22}] From Corollary \ref{Boshra33}, we have
  $$exp_r\left(\frac{\nu(1-\nu)}{2}\left(1-\frac{1}{h'}\right)^2\right)\T\sharp_{\nu}\S\leq \T\nabla_{\nu} \S$$
  for $\T,\S> 0$ with $m',M' > 0$  satisfying $0 < m' \leq \T \leq \S \leq M'$ or $0 < m' \leq \S \leq \T \leq M'$,
where $exp_r (\xi) := (1 + r\xi)^{1/r}$, if $1 + r\xi > 0$, and it is undefined otherwise. Since $exp_r(\xi)$ is
decreasing in $-1\leq r<0$, the above inequality gives a tight lower bound when $r$ equal $-1$. After
all, we have the scalar inequality:
$$\gamma \sqrt{\sigma\tau}\leq \frac{\sigma+\tau}{2}$$
for $\sigma, \tau > 0$ and $m',M' > 0$ such that $0 < m' \leq \min\{\sigma, \tau \} \leq \max\{\sigma, \tau \} \leq M'$. Applying this
inequality with a similar argument as in Theorem \ref{Prof1}, we obtain the desired result.
\end{proof}
\begin{theorem}\label{mohmmm}
   Let $\A,\B,\C,\D\in\bh$ $0<\nu<1$ and let $\psi$ be a non-negative increasing convex function on $[0,\infty)$. Then
   \begin{equation}\label{mohm1}
     \psi(w^2(\D\C\B\A))\leq \norm{(1-\nu) \psi\left(|\B\A|^{\frac{2}{1-\nu}}\right)+\nu \psi\left(|(\D\C)^*|^{\frac{2}{\nu}}\right)}-r\gamma(\psi)
   \end{equation}
   where $r=\min\{\nu,1-\nu\}$ and
   \begin{eqnarray}\label{mohm2}
     \gamma(\psi)&=&\inf_{\norm{\xi}=1}\left\{\psi\left(\seq{|\B\A|^{\frac{2}{1-\nu}}\xi,\xi}\right)
     +\psi\left(\seq{|(\D\C)^*|^{\frac{2}{\nu}}\xi,\xi}\right)\right.\nonumber \\
      &-&\left. 2\psi\left(\seq{\left(\frac{\seq{|\B\A|^{\frac{2}{1-\nu}}\xi,\xi}+\seq{|(\D\C)^*|^{\frac{2}{\nu}}\xi,\xi}}{2}\right)\xi,\xi}\right) \right\}.
   \end{eqnarray}
\end{theorem}
\begin{proof}
  We assume $0\leq \nu\leq \frac{1}{2}$. For each unit vector $\xi\in\h$,
  \begin{eqnarray*}
     && \psi\left(\seq{\left((1-\nu)|\B\A|^{\frac{2}{1-\nu}}+\nu |(\D\C)^*|^{\frac{2}{\nu}}\right)\xi,\xi}\right)+r\gamma(\psi) \\
     &=&\psi\left((1-\nu)\seq{|\B\A|^{\frac{2}{1-\nu}}\xi,\xi}+\nu\seq{|(\D\C)^*|^{\frac{2}{\nu}}\xi,\xi}\right)+r\gamma(\psi) \\
     &=& \psi\left((1-2\nu)\seq{|\B\A|^{\frac{2}{1-\nu}}\xi,\xi}+2\nu\seq{\left(\frac{|\B\A|^{\frac{2}
     {1-\nu}}+|(\D\C)^*|^{\frac{2}{\nu}}}{2}\right)\xi,\xi}\right)+r\gamma(\psi) \\
     &\leq&(1-2\nu)\psi\left(\seq{|\B\A|^{\frac{2}{1-\nu}}\xi,\xi}\right)+2\nu \psi\left(\seq{\left(\frac{|\B\A|^{\frac{2}{1-\nu}}+|(\D\C)^*|^{\frac{2}{\nu}}}{2}\right)\xi,\xi}\right) +r\gamma(\psi)\\
     &&(\mbox{by convexity of $\psi$}).
      \end{eqnarray*}
     Hence
     \begin{eqnarray*}
     && \psi\left(\seq{\left((1-\nu)|\B\A|^{\frac{2}{1-\nu}}+\nu |(\D\C)^*|^{\frac{2}{\nu}}\right)\xi,\xi}\right)+r\gamma(\psi) \\
     &\leq&(1-2\nu)\psi\left(\seq{|\B\A|^{\frac{2}{1-\nu}}\xi,\xi}\right)+2\nu \psi\left(\seq{\left(\frac{|\B\A|^{\frac{2}{1-\nu}}+|(\D\C)^*|^{\frac{2}{\nu}}}{2}\right)\xi,\xi}\right)  \\
     &+&  r\left(\psi\left(\seq{|\B\A|^{\frac{2}{1-\nu}}\xi,\xi}\right)+\psi\left(\seq{|(\D\C)^*|^{\frac{2}{\nu}}\xi,\xi}\right)\right.\\
   &-&\left. 2\psi\left(\seq{\left(\frac{|\B\A|^{\frac{2}{1-\nu}}
     +|(\D\C)^*|^{\frac{2}{\nu}}}{2}\right)\xi,\xi}\right)\right)\\
     &&\quad(\mbox{by inequality \ref{mohm2}})\\
      &\leq& (1-\nu)\psi\left(\seq{|\B\A|^{\frac{2}{1-\nu}}\xi,\xi}\right)+\nu \psi\left(\seq{|(\D\C)^*|^{\frac{2}{\nu}}\xi,\xi}\right)\\
     &\leq& \seq{\left((1-\nu)\psi\left(|\B\A|^{\frac{2}{1-\nu}}\right)+\nu \psi\left(|(\D\C)^*|^{\frac{2}{\nu}}\right)\right)\xi,\xi}\quad(\mbox{by Lemma \ref{yahya1}}).
  \end{eqnarray*}
  If we apply similar arguments for $\frac{1}{2}\leq \nu\leq 1$, then we can write
  \begin{eqnarray*}
   \psi\left(\seq{\left((1-\nu)|\B\A|^{\frac{2}{1-\nu}}+\nu |(\D\C)^*|^{\frac{2}{\nu}}\right)\xi,\xi}\right) &\leq& \left\|\left((1-\nu)\psi\left(|\B\A|^{\frac{2}{1-\nu}}\right)\right.\right. \\
     &+&\left.\left. \nu \psi\left(|(\D\C)^*|^{\frac{2}{\nu}}\right)\right)\xi,\xi\right\|-r\gamma(\psi).
  \end{eqnarray*}
  We know that if $\T\in\bh$ is a positive operator, then $\norm{\T}=\sup_{\norm{\xi}=1}\seq{\T\xi,\xi}$. By using this,
the continuity and the increase of $\psi$, we have
\begin{eqnarray*}
&&\psi\left(\norm{(1-\nu)|\B\A|^{\frac{2}{1-\nu}}+\nu |(\D\C)^*|^{\frac{2}{\nu}}}\right) \\
  &=&\psi\left(\sup_{\norm{x}=1}\seq{\left((1-\nu)|\B\A|^{\frac{2}{1-\nu}}+\nu |(\D\C)^*|^{\frac{2}{\nu}}\right)\xi,\xi}\right) \\
   &=&\sup_{\norm{x}=1}\psi\left(\seq{\left((1-\nu)|\B\A|^{\frac{2}{1-\nu}}+\nu |(\D\C)^*|^{\frac{2}{\nu}}\right)\xi,\xi}\right) \\
   &\leq& \sup_{\norm{x}=1}\seq{\left((1-\nu)\psi\left(|\B\A|^{\frac{2}{1-\nu}}\right)+\nu \psi\left(|(\D\C)^*|^{\frac{2}{\nu}}\right)\right)\xi,\xi}-r\gamma(\psi) \\
   &=&\norm{(1-\nu)\psi\left(|\B\A|^{\frac{2}{1-\nu}}\right)+\nu \psi\left(|(\D\C)^*|^{\frac{2}{\nu}}\right)}-r\gamma(\psi)
\end{eqnarray*}
On the other hand, if $X\in\bh$, and if $\psi$ is a non-negative increasing function on $[0,\infty)$, then
$\psi(\norm{\X})=\norm{\psi(|\X|)}$.\\
Now from the proof of Theorem \ref{Rahma1}, we have
\begin{eqnarray*}
    \psi(w^2(\D\C\B\A)) &\leq &\psi\left(\norm{(1-\nu)|\B\A|^{\frac{2}{1-\nu}}+\nu |(\D\C)^*|^{\frac{2}{\nu}}}\right) \\
     &\leq& \norm{(1-\nu) \psi\left(|\B\A|^{\frac{2}{1-\nu}}\right)+\nu \psi\left(|(\D\C)^*|^{\frac{2}{\nu}}\right)}-r\gamma(\psi).
  \end{eqnarray*}
  This completes the proof.
\end{proof}
Inequality (\ref{mohm1}) includes several numerical radius inequalities as special cases.
\begin{corollary}\label{maka1}
   Let $\T\in\bh$, $\alpha+\beta\geq 1$, $0<\nu<1$ and let $\psi$ be a non-negative increasing convex function on $[0,\infty)$. Then
   \begin{equation}\label{result1}
     \psi(w^2(\T|\T|^{\beta-1}\T|\T|^{\alpha-1}))\leq \norm{(1-\nu) \psi\left(|\T|^{\frac{2\beta}{1-\nu}}\right)+\nu \psi\left(|\T^*|^{\frac{2\alpha}{\nu}}\right)}-r\gamma(\psi)
   \end{equation}
   where $r$ and $\gamma(\psi)$ as in Theorem \ref{mohmmm}.
\end{corollary}
\begin{proof} Let $\T = U|\T|$ be the polar decomposition of the operator $\T$, where $U$ is partial isometry and the kernel
$\ker(U) = N(|\T|)$. If we take $\D=U, \C=|\T|^{\beta}, \B=U$ and $\A=|\T|^{\alpha}$, we have
$$\D\C\B\A=\T|\T|^{\beta-1}\T|\T|^{\alpha-1}, |\B\A|^2=|\T|^{2\alpha}\quad\mbox{and}\quad |(\D\C)^*|^2=|\T^*|^{2\beta}.$$
So, the result follows by Theorem \ref{mohmmm}.
\end{proof}
\begin{corollary}\label{maka2}
   Let $\T\in\bh$, $\alpha,\beta\geq 0$ such that $\alpha+\beta\geq 2$, $0<\nu<1$ and let $\psi$ be a non-negative increasing convex function on $[0,\infty)$. Then
   \begin{equation}\label{result2}
     \psi(w^2(\T^*|\T^*|^{\alpha+\beta-2}\T))\leq \norm{(1-\nu) \psi\left(|\T|^{\frac{2\beta}{1-\nu}}\right)+\nu \psi\left(|\T|^{\frac{2\alpha}{\nu}}\right)}-r\gamma(\psi)
   \end{equation}
   where  $r$ and $\gamma(\psi)$ as in Theorem \ref{mohmmm}.
\end{corollary}
\begin{proof}
  Let $\T^* = U|\T^*|$ be the polar decomposition of the operator $\T^*$, where $U$ is partial isometry and the kernel
$\ker(U) = N(|\T|)$. Then $\T=|\T^*|U^*$.
 If we take $\D=U, \C=|\T^*|^{\beta}, \B=|\T^*|^{\alpha}$ and $\A=U^*$, we have
$$\D\C\B\A=U|\T^*|^{\beta}|\T^*|^{\alpha}U^*=\T^*|\T^*|^{\alpha+\beta-2}\T, |\B\A|^2=|\T|^{2\alpha}\,\,\mbox{and}\,\, |(\D\C)^*|^2=|\T|^{2\beta}.$$
So, the result follows by Theorem \ref{mohmmm}.
\end{proof}
\begin{corollary}\label{maka3}
  Let $\T\in\bh$, $\alpha,\beta\geq 0$,
  $0<\nu<1$ and let $\psi$ be a non-negative increasing convex function on $[0,\infty)$. Then
  \begin{equation}\label{mohm1}
     \psi(w^2(|\T|^{\alpha}\T^2|\S|^{\beta}))\leq \norm{(1-\nu) \psi\left(|\T|^{\frac{2\beta+2}{1-\nu}}\right)+\nu \psi\left(|\T^*|\T|^{\alpha}|^{\frac{2}{\nu}}\right)}-r\gamma(\psi)
   \end{equation}
  where $r$ and $\gamma(\psi)$ as in Theorem \ref{mohmmm}.
\end{corollary}
\begin{proof} In Theorem \ref{mohmmm}, if we let
$\D=|\T|^{\alpha},\C=\T,\B=\T$ and $\A=|\T|^{\beta}$, then
\begin{eqnarray*}
  |\B\A|^2 &=& \A^*|\B|^2\A=|\T|^{\beta}|\T|^2|\T|^{\alpha}=|\T|^{2\beta+2} \\
   |(\D\C)^*|^2&=&\D|\C^*|^2\D^*=|\T|^{\alpha}|\T^*|^2|\T|^{\alpha}=|\T|^{\alpha}\T\T^*|\T|^{\alpha}\\
   &=&|\T|^{\alpha}\T(|\T|^{\alpha}\T)^*=|(|\T|^{\alpha}\T)^*|^2=|\T^*|\T|^{\alpha}|^2,
\end{eqnarray*}
   so the result
\end{proof}
Inequalities for numerical radius and operator norm have now been given, although in the context of superquadratic functions.
Remember that a function $\psi:[0,\infty) \to \R$ is termed superquadratic if there exists a constant $C_x\in \R$ such that
\begin{equation}\label{sq}
  \psi (t) \geq  \psi (\xi) + C_{\xi} (t-\xi) + \psi (|t-\xi|)
\end{equation}
 for all $t\geq 0$. If $-\psi$ is superquadratic, we say $\psi$ is subquadratic. As a result, for a superquadratic function,
  $\psi$ must be above its tangent line plus a translation of $\psi$.
 Superquadratic functions appear to be stronger than convex functions at first glance, however they may be deemed weaker if $\psi$ has negative values. If $\psi$ is superquadratic and non-negative, then Then $\psi$ is increasing and convex, and if $C_{\xi}$ is equal to (\ref{sq}), then $C_{\xi}\geq 0$ \cite{AJS}.
\begin{theorem}\label{supqad}
  Let $\A\in\bh$ and let $\psi$ be a non-negative superquadratic function. Then
  \begin{equation}\label{sq2}
    \psi(w(\A))\leq \norm{\psi(|\A|)}-\inf_{\norm{\xi}=1}\norm{\psi(||\A|-\norm{\A}|)^{\frac{1}{2}}}^2,
  \end{equation}
\end{theorem}
\begin{proof}
  Letting $\xi=\norm{\A}$ in the inequality (\ref{sq}), we get
  \begin{equation}\label{sq3}
    \psi(t)\geq \psi(\norm{\A})+C_{\norm{\A}}(t-\norm{\A})+\psi(|t-\norm{\A}|).
  \end{equation}
  By applying functional calculus for the operator $|\A|$ in (\ref{sq3}) we get
  \begin{equation}\label{sq4}
    \psi(|\A|)\geq \psi(\norm{\A})+C_{\norm{\A}}(|\A|-\norm{\A})+\psi(||\A|-\norm{\A}|)
  \end{equation}
  Hence,
  $$\seq{\psi(|\A|)\xi,\xi}\geq \psi(\norm{\A})+C_{\norm{\A}}(\seq{|\A|\xi,\xi}-\norm{\A})+\seq{\psi(||\A|-\norm{\A}|)\xi,\xi}.$$
  Consequently,
  \begin{equation}\label{sq5}
  \seq{\psi(|\A|)\xi,\xi}\geq \psi(\norm{\A})+C_{\norm{\A}}(\seq{|\A|\xi,\xi}-\norm{\A})+\norm{\psi(||\A|-\norm{\A}|)^{\frac{1}{2}}x}^{2}
  \end{equation}
  for every unit vector $\xi\in\h$.\\
  Now, by taking supremum over $\xi\in\h$ with $\norm{\xi}=1$ in (\ref{sq5}), and using the fact $w(|\A|)=\norm{\A}\geq w(\A)$,
  and $\psi$ is increasing, we deduce the desired inequality (\ref{sq2}).
\end{proof}
Applying Theorem \ref{supqad} to the superquadratic function $\psi(t)= t^r$ $(r \geq 2)$, we reach the
 following corollary:
 \begin{corollary}
   Let $A\in\bh$. Then for any $r\geq 2$,
$$w^r(\A)\leq \norm{\A}^r-\inf_{\norm{x}=1}\norm{||\A|-\norm{\A}|^{\frac{r}{2}}}^2.$$
In particular
$$w(\A)\leq\sqrt{\norm{\A}^2-\inf_{\norm{x}=1}\norm{||\A|-\norm{\A}|}^2}\leq \norm{\A}.$$
 \end{corollary}
\section{Further refinements of numerical radius inequalities}
In this section, We provide various inequalities involving power numerical radii  and the
usual operator norms of Hilbert space operators. In particular, if $\A_i, \B_i$ and $\X_i$ are bounded
linear operators ($i = 1,2,\cdots n\in\N$) , then we estimate the numerical radius
to $\sum_{j=1}^{m}\X_i\A_i^m\B_i$ for some $m\in \N$.\\
The following lemma is a straightforward application of Jensen's inequality about the convexity or concavity of certain power functions.
Schl\"omilch's inequality for the weighted means of non-negative real numbers is a specific example of this inequality.
\begin{lemma}\label{Jensen}
  Let $\sigma,\tau>0$ and $0\leq \alpha\leq 1$. Then
  \begin{equation}
    \sigma^{\alpha}\tau^{1-\alpha}\leq \alpha \sigma+(1-\alpha)\tau\leq \left(\alpha \sigma^r+(1-\alpha)\tau^r\right)^{\frac{1}{r}}\quad\mbox{for}\quad r\geq 1.
  \end{equation}
\end{lemma}
The following result was established by \cite{KM}, which is a refinement of the scalar Young inequality.
\begin{lemma}\label{alqad1}
  Let $\sigma,\tau>0$, and $p,q>1$ such that $\frac{1}{p}+\frac{1}{q}=1$. Then
  \begin{equation}\label{Eq2}
    \sigma\tau+r_0(\sigma^{\frac{p}{2}}-\tau^{\frac{q}{2}})^2\leq \frac{\sigma^p}{p}+\frac{\tau^q}{q},
  \end{equation}
  where $r_0=\min\{\frac{1}{p},\frac{1}{q}\}$.
\end{lemma}
Manasrah and Kittaneh  have generalized (\ref{Eq2}) in \cite{KM1}, as follows:
\begin{lemma}\label{alqad2}
  Let $\sigma,\tau>0$, and $p,q>1$ such that $\frac{1}{p}+\frac{1}{q}=1$. Then for $m=1,2,\cdots,$ we have
  \begin{equation}\label{Eq.3}
    (\sigma^{\frac{1}{p}}\tau^{\frac{1}{q}})^{m}+r_0^m(\sigma^{\frac{m}{2}}-\tau^{\frac{m}{2}})^2\leq \left(\frac{\sigma^r}{p}+\frac{\tau^r}{q}\right)^{\frac{m}{r}},\quad r\geq 1
  \end{equation}
where $r_0=\min\{\frac{1}{p},\frac{1}{q}\}$. In particular, if $p=q=2$, then
\begin{equation}\label{Eq.4}
  (\sqrt{\sigma\tau})^{m}+\frac{1}{2^m}(\sigma^{\frac{m}{2}}-\tau^{\frac{m}{2}})^2\leq 2^{-\frac{m}{r}}(\sigma^r+\tau^r)^{\frac{m}{r}}.
\end{equation}
For $m=1$, and $p=q=2$, we have
\begin{equation}\label{Eq.5}
  \sqrt{\sigma\tau}+\frac{1}{2}(\sqrt{\sigma}-\sqrt{\tau})^2\leq 2^{-\frac{1}{r}}(\sigma^r+\tau^r)^{\frac{1}{r}}.
\end{equation}
\end{lemma}
The convexity of the function $\psi(t)=t^r$, $r\geq 1$ leads to the following lemma, which deals with positive real numbers.
\begin{lemma}\label{Logain1}
  Let $\sigma_i,i=1,\cdots,n$ be positive real numbers. Then
  \begin{equation}\label{hoopy1}
    \left(\sum_{i=1}^{n}\sigma_i\right)^r\leq n^{r-1}\sum_{i=1}^{n}\sigma_i^r\quad \mbox{for}\,\,r\geq 1.
  \end{equation}
\end{lemma}
\begin{theorem}\label{thm3.1}
  Let $\A_i,\C_i,\\D_i\in\bh$, $(i=1,\cdots,n)$, $m\in\N$. Then
  \begin{equation}\label{Ineq.3.1}
    w^r\bra{\sum_{i=1}^{n}\D_i\C_i^{m}\A_i}\leq \frac{n^{r-1}}{2m}\sum_{j=1}^{m}\norm{\sum_{i=1}^{n}\bra{|\C_i^{j}\A_i|^{2r}+|(\D_i\C_i^{m-j})^*|^{2r}}}
  \end{equation}
  for all $r\geq 1$.
\end{theorem}
\begin{proof}
  Let $\xi\in\h$ be any unit vector. Then by Lemma \ref{D1}, Lemma \ref{Jensen} and Lemma
\ref{Logain1}, we obtain that
\begin{eqnarray*}
  \abs{\seq{\sum_{i=1}^{n}\D_i\C_i^{m}\A_i\xi,\xi} }^{r}&=& \frac{1}{m}\sum_{j=1}^{m}\abs{\seq{\sum_{i=1}^{n}\D_i\C_i^{m-j}\C_i^{j}\A_i\xi,\xi }}^{r} \\
   &\leq&\frac{1}{m}\sum_{j=1}^{m}\bra{\sum_{i=1}^{n}\abs{\seq{\D_i\C_i^{m-j}\C_i^{j}\A_i\xi,\xi}}}^{r}
   \end{eqnarray*}
  This implies that
    \begin{eqnarray*}
 \abs{\seq{\sum_{i=1}^{n}\D_i\C_i^{m}\A_i\xi,\xi} }^{r}  &\leq&\frac{n^{r-1}}{m}\sum_{j=1}^{m}\sum_{i=1}^{n}\abs{\seq{\D_i\C_i^{m-j}\C_i^{j}\A_i\xi,\xi}}^{r}  \\
  &\leq&\frac{n^{r-1}}{m}\sum_{j=1}^{m}\sum_{i=1}^{n}\seq{|\C_i^{j}\A_i|^2\xi,\xi}^{\frac{r}{2}}\seq{|(\D_i\C_i^{m-j})^*|^2\xi,\xi}^{\frac{r}{2}} \\
   &\leq&\frac{n^{r-1}}{m}\sum_{j=1}^{m}\sum_{i=1}^{n}\seq{|\C_i^{j}\A_i|^{2r}\xi,\xi}^{\frac{1}{2}}\seq{|(\D_i\C_i^{m-j})^*|^{2r}\xi,\xi}^{\frac{1}{2}}\\
   &\leq& \frac{n^{r-1}}{2m}\sum_{j=1}^{m}\sum_{i=1}^{n}\seq{\bra{|\C_i^{j}\A_i|^{2r}+|(\D_i\C_i^{m-j})^*|^{2r}}\xi,\xi}
\end{eqnarray*}
Taking the supremum over all unit vectors $\xi\in\h$, we get the result.
\end{proof}
For $\D_i=\A_i=I$ in inequality (\ref{Ineq.3.1}), we have
\begin{corollary}\label{cor.3.2}
  Let $\C_i\in\bh$, ($i=1,\cdots,n$), $m\in\N$. Then
  \begin{equation}\label{Ineq.3.2}
    w^r\bra{\sum_{i=1}^{n}\C_i^{m}}\leq \frac{n^{r-1}}{2m}\sum_{j=1}^{m}\norm{\sum_{i=1}^{n}\bra{|\C_i^{j}|^{2r}+|(\C_i^{m-j})^*|^{2r}}}
  \end{equation}
  for all $r\geq 1$.
\end{corollary}
The following is an example of how Corollary \ref{cor.3.2} may be used. It entails a numerical radius inequality for operator powers.
\begin{corollary}
  Let $\C\in\bh$ and $m\in\N$. Then for all $r\geq 1$, we have
  $$w^r(\C^m)\leq \frac{1}{2m}\sum_{j=1}^{m}\norm{|\C^j|^{2r}+|(\C^{m-j})^*|^{2r}}.$$
\end{corollary}
\begin{theorem}
  Let $\A_i,\C_i,\D_i\in\bh$, $(i=1,\cdots,n)$, $m\in\N$ and $0\leq \alpha\leq 1$. Then
  \begin{equation}\label{Ineq.3.1}
    w\bra{\sum_{i=1}^{n}\D_i\C_i^{m}\A_i}\leq
    \frac{1}{2m}\sum_{j=1}^{m}\sum_{i=1}^{n}\norm{\alpha|\C_i^{j}\A_i|^{\frac{2r}{\alpha}}
    +(1-\alpha)|(\D_i\C_i^{m-j})^*|^{\frac{2r}{1-\alpha}}}^{\frac{1}{2r}}
  \end{equation}
  for all $r\geq 1$.
\end{theorem}
\begin{proof} Let $\xi\in\h$ be any unit vector. Then by Lemma \ref{D1}, Lemma \ref{Jensen} and Lemma
\ref{Logain1}, we obtain
\begin{eqnarray*}
\abs{\seq{\sum_{i=1}^{n}\D_i\C_i^{m}\A_i\xi,\xi} } &=& \frac{1}{m}\sum_{j=1}^{m}\abs{\seq{\sum_{i=1}^{n}\D_i\C_i^{m-j}\C_i^{j}\A_i\xi,\xi }} \\
 &\leq&\frac{1}{m}\sum_{j=1}^{m}\sum_{i=1}^{n}\seq{|\C_i^{j}\A_i|^2\xi,\xi}^{\frac{1}{2}}\seq{|(\D_i\C_i^{m-j})^*|^2\xi,\xi}^{\frac{1}{2}}
 \end{eqnarray*}
 Hence
 \begin{eqnarray*}
  \abs{\seq{\sum_{i=1}^{n}\D_i\C_i^{m}\A_i\xi,\xi} } &\leq&\frac{1}{m}\sum_{j=1}^{m}\sum_{i=1}^{n}\bra{\seq{|\C_i^{j}\A_i|^2\xi,\xi}\seq{|(\D_i\C_i^{m-j})^*|^2\xi,\xi}}^{\frac{1}{2}}\\
   &\leq& \frac{1}{m}\sum_{j=1}^{m}\sum_{i=1}^{n}\bra{\seq{|\C_i^{j}\A_i|^{\frac{2}{\alpha}}\xi,\xi}^{\alpha}
   \seq{|(\D_i\C_i^{m-j})^*|^{\frac{2}{1-\alpha}}\xi,\xi}^{1-\alpha}}^{\frac{1}{2}}\\
   &\leq& \frac{1}{m}\sum_{j=1}^{m}\sum_{i=1}^{n}\bra{\alpha\seq{|\C_i^{j}\A_i|^{\frac{2}{\alpha}}\xi,\xi}^{r}+
   (1-\alpha)\seq{|(\D_i\C_i^{m-j})^*|^{\frac{2}{1-\alpha}}\xi,\xi}^{r}}^{\frac{1}{2r}}\\
   &\leq& \frac{1}{m}\sum_{j=1}^{m}\sum_{i=1}^{n}\seq{\bra{\alpha|\C_i^{j}\A_i|^{\frac{2r}{\alpha}}
   +(1-\alpha)|(\D_i\C_i^{m-j})^*|^{\frac{2r}{1-\alpha}}}\xi,\xi}^{\frac{1}{2r}}.
\end{eqnarray*}
Taking the supremum over all unit vectors $\xi\in\h$, we deduce the desired result.
\end{proof}
\begin{theorem}\label{N1}
  Let $\A_i,\C_i,\D_i\in\bh$, $(i=1,\cdots,n)$, $m\in\N$ and  $p,q>1$ such that $\frac{1}{p}+\frac{1}{q}=1$. Then
  \begin{equation}\label{N2}
     w^2\bra{\sum_{i=1}^{n}\D_i\C_i^{m}\A_i}\leq \frac{1}{m}\sum_{j=1}^{m}\norm{\sum_{i=1}^{n}\frac{1}{p}|\C_i^{j}\A_i|^{2p}+\frac{1}{q}|(\D_i\C_i^{m-j})^*|^{2q}}
    -r_0\inf_{\norm{x}=1}\psi(\xi),
  \end{equation}
 where $r_0=\min\{\frac{1}{p},\frac{1}{q}\}$ and
 \begin{equation*}
   \psi(\xi)=\frac{n}{m}\sum_{j=1}^{m}\sum_{i=1}^{n}\bra{\seq{||\C_i^{j}\A_i|^2\xi,\xi}^{\frac{p}{2}}-\seq{|(\D_i\C_i^{m-j})^*|\xi,\xi}^{\frac{q}{2}}}^2.
 \end{equation*}
\end{theorem}
\begin{proof}
  Let $\xi\in\h$ be any unit vector. Then by Lemma \ref{D1}, Lemma \ref{alqad1} and Lemma
\ref{Logain1}, we obtain
\begin{eqnarray*}
\abs{\seq{\sum_{i=1}^{n}\D_i\C_i^{m}\A_i\xi,\xi} }^2 &=& \frac{1}{m}\sum_{j=1}^{m}\abs{\seq{\sum_{i=1}^{n}\D_i\C_i^{m-j}\C_i^{j}\A_i\xi,\xi }}^2 \\
 &\leq&\frac{n}{m}\sum_{j=1}^{m}\sum_{i=1}^{n}\abs{\seq{\D_i\C_i^{m-j}\C_i^{j}\A_i\xi,\xi }}^2\\
 &\leq&\frac{n}{m}\sum_{j=1}^{m}\sum_{i=1}^{n}\seq{|\C_i^{j}\A_i|^2\xi,\xi}\seq{|(\D_i\C_i^{m-j})^*|^2\xi,\xi}\\
  &\leq& \frac{n}{m}\sum_{j=1}^{m}\sum_{i=1}^{n}\bra{\frac{1}{p}\seq{|\C_i^{j}\A_i|^{2p}\xi,\xi}+\frac{1}{q}\seq{|(\D_i\C_i^{m-j})^*|^{2q}\xi,\xi}}\\
  &-&\frac{nr_0}{m}\sum_{j=1}^{m}\sum_{i=1}^{n}\bra{\seq{|\C_i^{j}\A_i|^2\xi,\xi}^{\frac{p}{2}}-\seq{|(\D_i\C_i^{m-j})^*|^2\xi,\xi}^{\frac{q}{2}}}^2
   \end{eqnarray*}
   This implies that
 \begin{eqnarray*}
\abs{\seq{\sum_{i=1}^{n}\D_i\C_i^{m}\A_i\xi,\xi} }^2  &\leq& \frac{n}{m}\sum_{j=1}^{m}\sum_{i=1}^{n}\seq{\bra{\frac{1}{p}|\C_i^{j}\A_i|^{2p}+\frac{1}{q}|(\D_i\C_i^{m-j})^*|^{2q}}\xi,\xi}\\
  &-&\frac{nr_0}{m}\sum_{j=1}^{m}\sum_{i=1}^{n}\bra{\seq{|\C_i^{j}\A_i|^2\xi,\xi}^{\frac{p}{2}}-\seq{|(\D_i\C_i^{m-j})^*|^2\xi,\xi}^{\frac{q}{2}}}^2.
\end{eqnarray*}
Taking the supremum over all unit vectors $\xi\in\h$, we deduce the desired result.
\end{proof}
\begin{theorem}\label{N3}
  Let $\A_i,\C_i,\D_i\in\bh$, $(i=1,\cdots,n)$, $m\in\N$ and  $p,q>1$ such that $\frac{1}{p}+\frac{1}{q}=1$ and $k=1,2,\cdots$. Then
  \begin{equation}\label{N4}
     w^{2k}\bra{\sum_{i=1}^{n}\D_i\C_i^{m}\A_i}\leq \frac{n^{2k-1}}{m}\sum_{j=1}^{m}
     \sum_{i=1}^{n}\norm{\frac{1}{p}|\C_i^{j}\A_i|^{2rp}+\frac{1}{q}|(\D_i\C_i^{m-j})^*|^{2qr}}^{\frac{k}{r}}
    -r_0^{k}\inf_{\norm{x}=1}\eta(\xi),
  \end{equation}
 where $r_0=\min\{\frac{1}{p},\frac{1}{q}\}$ and
 \begin{equation*}
   \eta(\xi)=\frac{n^{2k-1}}{m}\sum_{j=1}^{m}\sum_{i=1}^{n}\bra{\seq{|\C_i^{j}\A_i|^{2p}\xi,\xi}^{\frac{k}{2}}-\seq{|(\D_i\C_i^{m-j})^*|^{2q}\xi,\xi}^{\frac{k}{2}}}^2
 \end{equation*}
  for all $r\geq 1$.
\end{theorem}
\begin{proof} Let $\xi\in\h$ be any unit vector. Then by Lemma \ref{D1}, Lemma \ref{alqad2} and Lemma
\ref{Logain1}, we obtain
\begin{eqnarray*}
 \abs{\seq{\sum_{i=1}^{n}\D_i\C_i^{m}\A_i\xi,\xi} }^{2k}&=& \frac{1}{m}\sum_{j=1}^{m}\abs{\seq{\sum_{i=1}^{n}\D_i\C_i^{m-j}\C_i^{j}\A_i\xi,\xi }}^{2k} \\
 &\leq&\frac{n^{2k-1}}{m}\sum_{j=1}^{m}\sum_{i=1}^{n}\abs{\seq{\D_i\C_i^{m-j}\C_i^{j}\A_i\xi,\xi}}^{2k} \\
 \end{eqnarray*}
 This implies that
 \begin{eqnarray*}
\abs{\seq{\sum_{i=1}^{n}\D_i\C_i^{m}\A_i\xi,\xi} }^{2k} &\leq&\frac{n^{2k-1}}{m}\sum_{j=1}^{m}\sum_{i=1}^{n}\bra{\seq{|\C_i^{j}\A_i|^2\xi,\xi}\seq{|(\D_i\C_i^{m-j})^*|^2\xi,\xi}}^k\\
  &\leq&\frac{n^{2k-1}}{m}\sum_{j=1}^{m}\sum_{i=1}^{n}\bra{\seq{|\C_i^{j}\A_i|^{\frac{2p}{p}}\xi,\xi}\seq{|(\D_i\C_i^{m-j})^*|^{\frac{2q}{q}}\xi,\xi}}^k\\
   &\leq&\frac{n^{2k-1}}{m}\sum_{j=1}^{m}\sum_{i=1}^{n}\bra{\seq{|\C_i^{j}\A_i|^{2p}\xi,\xi}^{\frac{1}{p}}\seq{|(\D_i\C_i^{m-j})^*|^{2q}\xi,\xi}^{\frac{1}{q}}}^k\\
    &\leq&\frac{n^{2k-1}}{m}\sum_{j=1}^{m}\sum_{i=1}^{n}\seq{\bra{\frac{1}{p}|\C_i^{j}\A_i|^{2rp}+\frac{1}{q}|(\D_i\C_i^{m-j})^*|^{2qr}}\xi,\xi}^{\frac{k}{r}}\\
    &-&\frac{n^{2k-1}r_0^k}{m}\sum_{j=1}^{m}\sum_{i=1}^{n}\bra{\seq{|\C_i^{j}\A_i|^{2p}\xi,\xi}^{\frac{k}{2}}-\seq{|(\D_i\C_i^{m-j})^*|^{2q}\xi,\xi}^{\frac{k}{2}}}^2
\end{eqnarray*}
Taking the supremum over all unit vectors $\xi\in\h$, we deduce the desired result.
\end{proof}
For $k=1$, and $p=q=2$, we have
\begin{corollary}
   Let $\A_i,\C_i,\D_i\in\bh$, $(i=1,\cdots,n)$, $m\in\N$.  Then
  \begin{equation}\label{N4}
     w^{2}\bra{\sum_{i=1}^{n}\D_i\C_i^{m}\A_i}\leq \frac{n 2^{-\frac{1}{r}}}{m}\sum_{j=1}^{m}
     \sum_{i=1}^{n}\norm{|\C_i^{j}\A_i|^{4r}+|(\D_i\C_i^{m-j})^*|^{4r}}^{\frac{1}{r}}
    -\frac{1}{2}\inf_{\norm{x}=1}\eta(x),
  \end{equation}
 where
 \begin{equation*}
   \eta(x)=\frac{n}{m}\sum_{j=1}^{m}\sum_{i=1}^{n}\bra{\seq{|\C_i^{j}\A_i|^{4}\xi,\xi}^{\frac{1}{2}}-\seq{|(\D_i\C_i^{m-j})^*|^{4}\xi,\xi}^{\frac{1}{2}}}^2
 \end{equation*}
  for all $r\geq 1$.
\end{corollary}
The following lemma is an extended variant of the mixed Schwarz inequality, which has been
 shown by  \cite{kit1} and is highly relevant in the following results.
\begin{lemma}\label{M-Cauchy}
  Let $\A\in\bh$, and $\psi$ and $\phi$ be non-negative functions on $[0,\infty)$
  which are continuous such that $\psi(t)\phi(t)=t$ for all $t\in [0,\infty)$. Then
  \begin{equation}\label{C11}
    \abs{\seq{\A\xi,\zeta}}\leq \norm{\psi(|\A|)\xi}\norm{\phi(|\A^*|)\zeta},
  \end{equation}
  for all $\xi,\zeta\in\h$.
\end{lemma}
The next results give improvements of the inequality (\ref{watfa2}).
\begin{theorem}\label{sch1}
  Let $\A_i,\B_i,\X_i\in\bh$, $(i=1,\cdots,n)$, $m\in\N$, $p,q>1$ with $\frac{1}{p}+\frac{1}{q}=1$ and let $\psi$ and $\phi$ be as in
  Lemma \ref{M-Cauchy}. Then for all $r\geq 1$, we have
  \begin{equation}\label{C1}
    w^{2r}\bra{\sum_{i=1}^{n}\X_i\A_i^m\B_i}\leq \frac{n^{2r-1}}{m}\sum_{j=1}^{m}\norm{\sum_{i=1}^{n}\frac{1}{p}S_{i,j}^{pr}+\frac{1}{q}T_{i,j}^{qr}}-r_0\inf_{\norm{x}=1}\rho(\xi),
  \end{equation}
  where $r_0=\min\{\frac{1}{p},\frac{1}{q}\}$, $S_{i,j}=\X_i\psi^2\bra{\abs{\A_i^{j*}}}\X_i^*$, $T_{i,j}=\bra{\A_i^{m-j}\B_i}^*\phi^2\bra{\abs{\A_i^j}}\A_i^{m-j}\B_i$ and
  $$\rho(\xi)=\frac{n^{2r-1}}{m}\sum_{j=1}^{m}\sum_{i=1}^{n}\bra{\seq{S_{i,j}^r\xi,\xi}^{\frac{p}{2}}-\seq{T_{i,j}^r\xi,\xi}^{\frac{q}{2}}}^2.$$
\end{theorem}
\begin{proof}
  Let $\xi\in\h$ be any unit vector. Then by Lemma \ref{alqad2}, Lemma
\ref{Logain1} and Lemma \ref{M-Cauchy},  we obtain
\begin{eqnarray*}
  \abs{\sum_{i=1}^{n}\seq{\X_i\A_i^m\B_i\xi,\xi}}^{2r} &=& \frac{1}{m}\sum_{j=1}^{m}\abs{\sum_{i=1}^{n}\seq{\X_i\A_i^{m-j}\A_i^{j}\B_i\xi,\xi}}^{2r} \\
   &\leq&\frac{1}{m}\sum_{j=1}^{m}\bra{\sum_{i=1}^{n}\abs{\seq{\X_i\A_i^{m-j}\A_i^{j}\B_i\xi,\xi}}}^{2r} \\
  &\leq&\frac{n^{2r-1}}{m}\sum_{j=1}^{m}\sum_{i=1}^{n}\abs{\seq{\X_i\A_i^{m-j}\A_i^{j}\B_i\xi,\xi}}^{2r}
   \end{eqnarray*}
  and so
 \begin{eqnarray*}
 \abs{\sum_{i=1}^{n}\seq{\X_i\A_i^m\B_i\xi,\xi}}^{2r}  &\leq&\frac{n^{2r-1}}{m}\sum_{j=1}^{m}\sum_{i=1}^{n}\abs{\seq{\A_i^{j*}\X_i^{*}x,\A_i^{m-j}\B_ix}}^{2r}\\
  &\leq&\frac{n^{2r-1}}{m}\sum_{j=1}^{m}\sum_{i=1}^{n}\norm{\psi\bra{\abs{\A_i^{j*}}}\X_i^*x}^{2r}\norm{\phi\bra{\abs{\A_i^j}}\A_i^{m-j}\B_ix}^{2r}\\
  &\leq&\frac{n^{2r-1}}{m}\sum_{j=1}^{m}\sum_{i=1}^{n}\seq{S_{i,j}\xi,\xi}^{r}\seq{T_{i,j}\xi,\xi}^{r}\\
   &\leq&\frac{n^{2r-1}}{m}\sum_{j=1}^{m}\sum_{i=1}^{n}\seq{S_{i,j}^r\xi,\xi}\seq{T_{i,j}^r\xi,\xi}
 \end{eqnarray*}
 Hence
 \begin{eqnarray*}
 \abs{\sum_{i=1}^{n}\seq{\X_i\A_i^m\B_i\xi,\xi}}^{2r}  &\leq&\frac{n^{2r-1}}{m}\sum_{j=1}^{m}\sum_{i=1}^{n}\seq{\bra{\frac{1}{p}S_{i,j}^{pr}+\frac{1}{q}T_{i,j}^{qr}}\xi,\xi}\\
  &-&r_0\frac{n^{2r-1}}{m}\sum_{j=1}^{m}\sum_{i=1}^{n}\bra{\seq{S_{i,j}^r\xi,\xi}^{\frac{p}{2}}-\seq{T_{i,j}^r\xi,\xi}^{\frac{q}{2}}}^2.
\end{eqnarray*}
Taking the supremum over all unit vectors $\xi\in\h$, we deduce the desired result.
\end{proof}
Inequality (\ref{C1}) includes several numerical radius inequalities as special cases.
Samples of inequalities are demonstrated in what follows, for $\psi(t)=t^{\lambda}$ and $\phi(t)=t^{1-\lambda}$,
$\lambda\in (0,1)$ in inequality (\ref{C1}).
\begin{corollary}\label{okasha1}
  Let $\A_i,\B_i,\X_i\in\bh$, $(i=1,\cdots,n)$, $m\in\N$, $p,q>1$ with $\frac{1}{p}+\frac{1}{q}=1$ and let $\psi$ and $g$ be as in
  Lemma \ref{M-Cauchy}. Then for all $r\geq 1$, we have
  \begin{equation}\label{C1}
    w^{2r}\bra{\sum_{i=1}^{n}\X_i\A_i^m\B_i}\leq \frac{n^{2r-1}}{m}\sum_{j=1}^{m}\norm{\sum_{i=1}^{n}\frac{1}{p}S_{i,j}^{pr}+\frac{1}{q}T_{i,j}^{qr}}-r_0\inf_{\norm{x}=1}\rho(\xi),
  \end{equation}
  where $r_0=\min\{\frac{1}{p},\frac{1}{q}\}$, $S_{i,j}=\X_i\abs{\A_i^{j*}}^{2\lambda}\X_i^*$,
  $T_{i,j}=\bra{\A_i^{m-j}\B_i}^*\abs{\A_i^j}^{2(1-\lambda)}\A_i^{m-j}\B_i$ and
  $$\rho(\xi)=\frac{n^{2r-1}}{m}\sum_{j=1}^{m}\sum_{i=1}^{n}\bra{\seq{S_{i,j}^r\xi,\xi}^{\frac{p}{2}}-\seq{T_{i,j}^r\xi,\xi}^{\frac{q}{2}}}^2.$$
\end{corollary}
For $\X_i=\B_i=I$ in inequality (\ref{C1}) we get the following numerical radius inequality
\begin{corollary}\label{okasha2}
  Let $\A_i,\in\bh$, $(i=1,\cdots,n)$, $m\in\N$, $p,q>1$ with $\frac{1}{p}+\frac{1}{q}=1$ and let $\psi$ and $g$ be as in
  Lemma \ref{M-Cauchy}. Then for all $r\geq 1$, we have
  \begin{equation}\label{C1}
    w^{2r}\bra{\sum_{i=1}^{n}\A_i^m}\leq \frac{n^{2r-1}}{m}\sum_{j=1}^{m}\norm{\sum_{i=1}^{n}\frac{1}{p}S_{i,j}^{pr}+\frac{1}{q}T_{i,j}^{qr}}-r_0\inf_{\norm{x}=1}\rho(\xi),
  \end{equation}
  where $r_0=\min\{\frac{1}{p},\frac{1}{q}\}$, $S_{i,j}=\psi^2\bra{\abs{\bra{\A_i^{j}}^*}}$,
  $T_{i,j}=\bra{\A_i^{m-j}}^{*}\phi^2\bra{\abs{\A_i^j}}\A_i^{m-j}$ and
  $$\rho(\xi)=\frac{n^{2r-1}}{m}\sum_{j=1}^{m}\sum_{i=1}^{n}\bra{\seq{S_{i,j}^r\xi,\xi}^{\frac{p}{2}}-\seq{T_{i,j}^r\xi,\xi}^{\frac{q}{2}}}^2.$$
\end{corollary}
An application of Corollary \ref{okasha2} can be seen in the following result. It involves a
numerical radius inequality for the powers of operator.
\begin{corollary}\label{okasha3}
  Let $A,\in\bh$,  $m\in\N$, $p,q>1$ with $\frac{1}{p}+\frac{1}{q}=1$ and let
  $\psi(t)=t^{\lambda}$ and $\phi(t)=t^{1-\lambda}$. Then for all $r\geq 1$, we have
  \begin{equation}\label{C1}
    w^{2r}\bra{A^m}\leq \frac{1}{m}\sum_{j=1}^{m}\norm{\frac{1}{p}S_{j}^{pr}+\frac{1}{q}T_{j}^{qr}}-r_0\inf_{\norm{x}=1}\rho(\xi),
  \end{equation}
  where $r_0=\min\{\frac{1}{p},\frac{1}{q}\}$, $S_{j}=\abs{\bra{A^{j}}^*}^{2\lambda}$,
  $T_{j}=\bra{A^{m-j}}^{*}\abs{A^j}^{2(1-\lambda)}A^{m-j}$ and
  $$\rho(\xi)=\frac{1}{m}\sum_{j=1}^{m}\bra{\seq{S_{j}^r\xi,\xi}^{\frac{p}{2}}-\seq{T_{j}^r\xi,\xi}^{\frac{q}{2}}}^2.$$
\end{corollary}
\begin{theorem}\label{sch2}
  Let $\A_i,\B_i,\X_i\in\bh$, $(i=1,\cdots,n)$, $m,k\in\N$, $p,q>1$ with $\frac{1}{p}+\frac{1}{q}=1$ and let $\psi$ and $\phi$ be as in
  Lemma \ref{M-Cauchy}. Then for all $r\geq 1$, we have
  \begin{equation}\label{C2}
    w^{2k}\bra{\sum_{i=1}^{n}\X_i\A_i^m\B_i}\leq \frac{n^{2k-1}}{m}\sum_{j=1}^{m}\sum_{i=1}^{n}\norm{\frac{1}{p}S_{i,j}^{pr}+\frac{1}{q}T_{i,j}^{qr}}^{\frac{k}{r}}-r_0^k\inf_{\norm{x}=1}\omega(\xi),
  \end{equation}
  where $r_0=\min\{\frac{1}{p},\frac{1}{q}\}$, $S_{i,j}=\X_i\psi^2\bra{\abs{\A_i^{j*}}}\X_i^*$, $T_{i,j}=\bra{\A_i^{m-j}\B_i}^*\phi^2\bra{\abs{\A_i^j}}\A_i^{m-j}\B_i$ and
  $$\omega(\xi)=\frac{n^{2k-1}}{m}\sum_{j=1}^{m}\sum_{i=1}^{n}\bra{\seq{S_{i,j}^{p}\xi,\xi}^{\frac{k}{2}}-\seq{T_{i,j}^{q}\xi,\xi}^{\frac{k}{2}}}^2.$$
\end{theorem}
\begin{proof}
  Let $\xi\in\h$ be any unit vector. Then by Lemma \ref{alqad2}, Lemma
\ref{Logain1} and Lemma \ref{M-Cauchy},  we obtain
\begin{eqnarray*}
  \abs{\sum_{i=1}^{n}\seq{\X_i\A_i^m\B_i\xi,\xi}}^{2k} &=& \frac{1}{m}\sum_{j=1}^{m}\abs{\sum_{i=1}^{n}\seq{\X_i\A_i^{m-j}\A_i^{j}\B_i\xi,\xi}}^{2k} \\
   &\leq&\frac{1}{m}\sum_{j=1}^{m}\bra{\sum_{i=1}^{n}\abs{\seq{\X_i\A_i^{m-j}\A_i^{j}\B_i\xi,\xi}}}^{2k} \\
  &\leq&\frac{n^{2k-1}}{m}\sum_{j=1}^{m}\sum_{i=1}^{n}\abs{\seq{\X_i\A_i^{m-j}\A_i^{j}\B_i\xi,\xi}}^{2k} \\
  &\leq&\frac{n^{2k-1}}{m}\sum_{j=1}^{m}\sum_{i=1}^{n}\abs{\seq{\A_i^{j*}\X_i^{*}x,\A_i^{m-j}\B_ix}}^{2k}
   \end{eqnarray*}
   this implies that
 \begin{eqnarray*}
  \abs{\sum_{i=1}^{n}\seq{\X_i\A_i^m\B_i\xi,\xi}}^{2k}&\leq&\frac{n^{2k-1}}{m}\sum_{j=1}^{m}\sum_{i=1}^{n}\norm{\psi\bra{\abs{\A_i^{j*}}}\X_i^*x}^{2k}\norm{\phi\bra{\abs{\A_i^j}}\A_i^{m-j}\B_ix}^{2k}\\
  &\leq&\frac{n^{2k-1}}{m}\sum_{j=1}^{m}\sum_{i=1}^{n}\bra{\seq{S_{i,j}^{p}\xi,\xi}^{\frac{1}{p}}\seq{T_{i,j}^{q}\xi,\xi}^{\frac{1}{q}}}^{k}\\
   &\leq&\frac{n^{2k-1}}{m}\sum_{j=1}^{m}\sum_{i=1}^{n}\bra{\seq{\frac{1}{p}S_{i,j}^{pr}\xi,\xi}+\frac{1}{q}\seq{T_{i,j}^{qr}\xi,\xi}}^{\frac{k}{r}}\\
   &-&r_0^k\frac{n^{2k-1}}{m}\sum_{j=1}^{m}\sum_{i=1}^{n}\bra{\seq{S_{i,j}^{p}\xi,\xi}^{\frac{k}{2}}-\seq{T_{i,j}^{q}\xi,\xi}^{\frac{k}{2}}}^2
 \end{eqnarray*}
 Hence
 \begin{eqnarray*}
 \abs{\sum_{i=1}^{n}\seq{\X_i\A_i^m\B_i\xi,\xi}}^{2k}  &\leq&\frac{n^{2r-1}}{m}\sum_{j=1}^{m}\sum_{i=1}^{n}\seq{\bra{\frac{1}{p}S_{i,j}^{pr}+\frac{1}{q}T_{i,j}^{qr}}\xi,\xi}^{\frac{k}{r}}\\
  &-&r_0^k\frac{n^{2k-1}}{m}\sum_{j=1}^{m}\sum_{i=1}^{n}\bra{\seq{S_{i,j}^{p}\xi,\xi}^{\frac{k}{2}}-\seq{T_{i,j}^{q}\xi,\xi}^{\frac{k}{2}}}^2.
\end{eqnarray*}
Taking the supremum over all unit vectors $\xi\in\h$, we deduce the desired result.
\end{proof}
If we take $k=1$ and $p=q$, we have
\begin{corollary}
   Let $\A_i,\B_i,\X_i\in\bh$, $(i=1,\cdots,n)$, $m\in\N$, and let $\psi$ and $\phi$ be as in
  Lemma \ref{M-Cauchy}. Then for all $r\geq 1$, we have
  \begin{equation}\label{C2}
    w^{2}\bra{\sum_{i=1}^{n}\X_i\A_i^m\B_i}\leq \frac{n}{m 2^{\frac{1}{r}}}\sum_{j=1}^{m}\sum_{i=1}^{n}\norm{S_{i,j}^{2r}+T_{i,j}^{2r}}^{\frac{1}{r}}-\frac{1}{2}\inf_{\norm{x}=1}\omega(\xi),
  \end{equation}
  where  $S_{i,j}=\X_i\psi^2\bra{\abs{\A_i^{j*}}}\X_i^*$, $T_{i,j}=\bra{\A_i^{m-j}\B_i}^*\phi^2\bra{\abs{\A_i^j}}\A_i^{m-j}\B_i$ and
  $$\omega(\xi)=\frac{n}{m}\sum_{j=1}^{m}\sum_{i=1}^{n}\bra{\seq{S_{i,j}^{2}\xi,\xi}^{\frac{1}{2}}-\seq{T_{i,j}^{2}\xi,\xi}^{\frac{1}{2}}}^2.$$
\end{corollary}

\bibliographystyle{unsrtnat}
\bibliography{references}  






\end{document}